\documentclass[reqno]{amsart}
\usepackage{cases}
\usepackage{latexsym}
\usepackage{amsmath}
\usepackage[arrow,matrix]{xy}
\usepackage{stmaryrd}
\usepackage{amsfonts}
\usepackage{amsmath,amssymb,amscd,bbm,amsthm,mathrsfs,dsfont}

\usepackage{fancyhdr}
\usepackage{amsxtra,ifthen}
\usepackage{verbatim}

\theoremstyle{plain}
\newtheorem{theorem}{Theorem}[section]
\newtheorem{lemma}[theorem]{Lemma}

\newtheorem{corollary}[theorem]{Corollary}

\theoremstyle{definition}

\newtheorem{remark}[theorem]{Remark}

\begin{document}

\title[Commuting maps of inflated algebras]
{Commuting maps of inflated algebras}

%
%
%
%
%
%

\author{Hongyu Jia$^{1}$}
\address{Jia: School of Mathematical Sciences, Huaqiao University,
Quanzhou, Fujian, 362021, P. R. China}
\email{jiahy1995@163.com }

\author{Zhankui Xiao$^{1,}$* }
\thanks{*Corresponding Author. Email address: zhkxiao@hqu.edu.cn}
\address{Xiao: School of Mathematical Sciences, Huaqiao University,
Quanzhou, Fujian, 362021, P. R. China}
\email{zhkxiao@hqu.edu.cn}

\thanks{$^{1}$ School of Mathematical Sciences, Huaqiao University,
Quanzhou, Fujian, 362021, P. R. China}

\begin{abstract}
Commuting maps on a class of algebras called inflated
algebras are investigated. 
In particular, we can prove that  every
commuting map $\theta$ on such an algebra is of the form
$\theta(x)=c x+\mu(x)$, where $c$ belongs to the base
field $K$ of characteristic not 2, and $\mu$ is a central-valued linear map.
\end{abstract}

\subjclass[2020]{Primary 16W10, Secondary 17B40, 17B60}

\keywords{commuting map, inflated  algebra}

\maketitle

\section{Introduction}\label{xxsec1}

Throughout this paper, we assume that $K$ is a field of $\operatorname{char} 
(K)\neq 2$. Let $\mathcal{A}$ be an associative algebra over $K$ (not necessarily unital). Then $\mathcal{A}$ becomes a Lie algebra with respect to the Lie bracket operation $[x, y]:=xy-yx$.
A $K$-linear map $\theta: \mathcal{A}\rightarrow \mathcal{A}$ is called {\em commuting} if $[\theta(x), x]=0$ for all $x\in \mathcal{A}$.
A commuting map $\theta$ of $\mathcal{A}$ is said to be {\em proper}
if it is of the form 
$$
\theta(x)=\lambda x+\mu(x), \quad \forall x\in \mathcal{A},
$$
where $\lambda\in Z(\mathcal{A})$, the center of $\mathcal{A}$, and $\mu$ is a $K$-linear map with range in $Z(\mathcal{A})$. A commuting map that is not proper will be called {\em improper}.
If the associative algebra $\mathcal{A}$ is not unital, it is clear that,
for any scalar $c\in K$, the map of $\mathcal{A}$ defined by $x \mapsto cx$ is a typical example of commuting maps. Hence we call a commuting map $\theta$ of $\mathcal{A}$ {\em standard} if it is of the form
$$
\theta(x)=cx+\lambda x +\mu(x), \quad \forall x\in \mathcal{A},
$$
where $c\in K$, $\lambda\in Z(\mathcal{A})$ and $\mu$ is a $K$-linear map with
range in $Z(\mathcal{A})$. For a unital associative algebra $\mathcal{A}$, the notion of a proper commuting map coincides with that of a standard one. If $\mathcal{A}$ is an associative algebra without unity, we denote by $\hat{\mathcal{A}}$ the unitization of $\mathcal{A}$. Then a similar proof of
\cite[Proposition 1]{Cheung} shows that $\mathcal{A}$ has no non-standard commuting map if and only if $\hat{\mathcal{A}}$ has no improper commuting map.

To our knowledge, the first important results on commuting maps are due to
Divinsky \cite{Divinsky} and Posner \cite{Posner}.
The renowned Posner's second theorem states that the existence of a nonzero commuting derivation on a prime ring $\mathcal{A}$ implies the commutativity of $\mathcal{A}$ (see \cite[Theorems 1 and 2]{Posner}).
Motivated by Posner's work, Bre\v{s}ar in \cite{Bre93-1} studied general additive commuting maps on prime rings.
Furthermore, in \cite{Bre93-2} he found an intrinsic connection between commuting maps and the Herstein's Lie-type mapping research program (see \cite{Her}), which in turn promoted the theory of functional identities.
We encourage the reader to read the well-written survey paper \cite{Bre04}, in which Bre\v{s}ar outlined the theory of commuting maps, and particular emphasis was placed
on its application in Lie theory.
More results related to commuting maps can be found in
\cite{Bre91, Bre94, ChenCai, fo1, Lee, LeeLee, XY}, etc.

It was Cheung \cite{Cheung} who initiated the study of commuting maps on formal matrix algebras, and
he determined a class of triangular algebras on which every commuting map is proper. 
The second author of this paper and Wei \cite{XiaoWei} extended
the main results of \cite{Cheung} to
generalized matrix algebras, a kind of Morita context rings.
Recently, the second and third authors of this paper in \cite{JX} established a sufficient and necessary condition on
incidence algebras such that every commuting map is proper.
Motivated by the aforementioned results, we will
study commuting maps on inflated algebras, which are deformations of the full matrix algebras ${\rm M}_n(K)$ and
play an important role in the representation theory of classical groups.

The paper is organized as follows. In Section 2, we introduce the notion of inflated algebras $\mathfrak{M}$ and present some properties of $\mathfrak{M}$ for later use. Section 3 is devoted to the study of commuting maps on $\mathfrak{M}$. By an explicitly description of the action of a commuting map on a basis of $\mathfrak{M}$, we can prove that every commuting map of an inflated algebra is standard. Here we also extend our results to a broader class of Munn's semigroup algebras.

\section{Inflated algebras}\label{xxsec2}

In this section, we recall the definition and some basic properties of inflated algebras.
Let $K$ be a field and $V$ be a $K$-linear space.
Given a $K$-bilinear form $\gamma: V \times V \to K$,
we define an associative algebra  $\mathfrak{M}=\mathfrak{M}(V, \gamma)$ as follows:
as a $K$-linear space, $\mathfrak{M}$ is equal to $V \otimes_K V$ and the multiplication is defined by
$$
(a \otimes b) \cdot (c \otimes d) := \gamma (b, c) a\otimes d, \mathds{R} 
$$
for all $ a, b, c, d \in V $.
This definition makes $\mathfrak{M}$ become an associative $K$-algebra and $\mathfrak{M}$ is called the \textit{inflated algebra of $K$ along $V$}.

Inflated algebras appeared naturally in the representation theory of classical groups and quantum groups (see \cite{xi1} and the references therein).
In fact, \cite[Theorem 4.1]{xi1} shows that every cellular algebra (for instance, Hecke algebras of finite Coxeter groups, Brauer algebras, etc) 
over $K$ is an iterated inflation of finitely many copies of $K$.

Let ${\rm M}_n(K)$ be the set of all $n\times n$ matrices over $K$.
Since $V \otimes_K V\cong {\rm M}_n(K)$ as vector spaces,
where $n={\rm dim}(V)$, we can realize
the structure of an inflated algebra in ${\rm M}_n(K)$.
Let $\{v_1, v_2, \cdots, v_n\}$ be a basis of $V$.
Then the bilinear form $\gamma$ can be characterized by an $n\times n$
matrix $P$ over $K$, that is $P= (\gamma(v_i, v_j)) $ for $1 \leq i, j \leq n$. Now we define a new multiplication ‘$\cdot$’ on the set
${\rm M}_n(K)$ by
$$
A \cdot B =APB \quad \text{for all } A,B \in {\rm M}_n(K).
$$
Under the usual linear operations and the multiplication $\cdot$\, ,
${\rm M}_n(K)$ becomes an associative $K$-algebra,
denoted by $({\rm M}_n(K), P)$. It can be shown that the inflated algebra $\mathfrak{M}(V, \gamma)$ is isomorphic to $({\rm M}_n(K), P)$
(see \cite[Lemma 4.1]{xi2}).
Notice that the inflated algebra $({\rm M}_n(K), P)$ is a generalized matrix algebra in the sense of Brown \cite{Brown}, and is also a Munn's semigroup
algebra in the sense of Munn \cite{Munn}.
The following technical lemma is to some extent well-known
and we sketch the proof here for reader's convenience.

\begin{lemma}\label{lemma1}
With notations as above, if $A$ and $B$ are invertible $n\times n$ matrices over $K$, then $({\rm M}_n(K), P)$ and $({\rm M}_n(K), APB)$ are isomorphic as algebras.
\end{lemma}
\begin{proof}
The map $\phi:({\rm M}_n(K), P)\to ({\rm M}_n(K), APB)$
defined by $\phi(X)=B^{-1} X A^{-1}$
is an isomorphism of algebras.
\end{proof}

Given an inflated algebra $\mathfrak{M}(V, \gamma)$ and an arbitrary basis $\{v_1, v_2, \cdots, v_n\}$ of $V$, if the rank of $P= (\gamma(v_i, v_j))$ is $r$,
then $\mathfrak{M}(V, \gamma)\cong ({\rm M}_n(K), 
\left[\smallmatrix I_{r} & 0 \\ 0 & 0 \endsmallmatrix \right] )$ as algebras by Lemma \ref{lemma1},
where $I_{r}$ is the $r \times r$ identity matrix and we
use the convention $I_0=0$ if $r=0$.

From now on, we set $\mathfrak{M}_n(r):=({\rm M}_n(K), 
\left[\smallmatrix I_{r} & 0 \\ 0 & 0 \endsmallmatrix \right] )$ for convenience. If $r=0$, the multiplication of
$\mathfrak{M}_n(r)$ is trivial, i.e., $A \cdot B=0$ for all $A, B\in \mathfrak{M}_n(r)$. If $r=n$, the inflated algebra
$\mathfrak{M}_n(r)$ is just the full matrix algebra
${\rm M}_n(K)$. For $0\leq r\leq n$, we define
$\mathcal{J}= \{ 1,\cdots,r\}$ and $\hat{\mathcal{J}} = \{r+1,\cdots,n\}$. Then $\mathbf{n}:=\{1,\cdots,n\}=\mathcal{J} \sqcup \hat{\mathcal{J}}$.
For all $A \in \mathfrak{M}_n(r)$ and $i,j \in \mathbf{n}$,
we denote by $e_{ij}$ the usual matrix unit and have
$$
e_{ij} \cdot  A = 
\begin{cases} 
e_{ij}   A & \text{if } j\in \mathcal{J}, \\
0 & \text{if } j \in \hat{\mathcal{J}},
\end{cases}  \quad 
A \cdot  e_{ij} =
\begin{cases} 
 A e_{ij}   & \text{if } i\in \mathcal{J}, \\
0 & \text{if } i \in \hat{\mathcal{J}},
\end{cases}
$$
and
\[
e_{ij} \cdot  e_{kl} = 
\begin{cases} 
\delta_{jk}e_{il} & \text{if } j, k \in \mathcal{J}, \\
0 & \text{otherwise},
\end{cases}
\]
where $\delta_{jk}\in \{0,1\}$ is the Kronecker delta. These facts related to the multiplication of the inflated algebra
$\mathfrak{M}_n(r)$ will be frequently used later.
To avoid confusion, here we would like to emphasize that
the Lie bracket of $\mathfrak{M}_n(r)$ is defined by
\[
[A, B] = A \cdot  B - B \cdot  A, \quad \text{for all } A, B \in \mathfrak{M}_n(r).
\]

We end this section by a characterization of the center of
$\mathfrak{M}_n(r)$.

\begin{lemma}\label{l1}
Let $Z(\mathfrak{M}_n(r))$ be the center of the inflated algebra $\mathfrak{M}_n(r)$.
\begin{enumerate}
 \item[(i)] If $\hat{\mathcal{J}} =\emptyset$, then $Z(\mathfrak{M}_n(r)) = \left\{ \lambda I_n \mid \lambda \in K \right\}$.
 \item[(ii)] If $\hat{\mathcal{J}} \neq \emptyset$, then $Z(\mathfrak{M}_n(r)) =K\text{-}{\rm Span} \left\{ e_{ij} \mid i, j \in \hat{\mathcal{J}} \right\}$.
\end{enumerate}
\end{lemma}

\begin{proof}
We only need to prove the second statement.
For $\hat{\mathcal{J}} \neq \emptyset$, if $\mathcal{J} = \emptyset$, this result is obvious.
We then assume that $\mathcal{J} \neq \emptyset$, i.e.,
$0<r<n$. 
Given an arbitrary matrix $A\in \mathfrak{M}_n(r)$, we rewrite
it in block form
\[
A = \begin{bmatrix} A_{11} & A_{12} \\ A_{21} & A_{22} \end{bmatrix},
\]
where $A_{11}$ is an $r\times r$ matrix.
If
\[
A = \begin{bmatrix} A_{11} & A_{12} \\ A_{21} & A_{22} \end{bmatrix} \in Z(\mathfrak{M}_n(r)) \quad \text{and} \quad B = \begin{bmatrix} B_{11} & B_{12} \\ B_{21} & B_{22} \end{bmatrix} \in \mathfrak{M}_n(r),
\]
we have 
\[
0 = [A, B] = \begin{bmatrix} 
A_{11}B_{11} - B_{11}A_{11} & A_{11}B_{12} - B_{11}A_{12} \\
A_{21}B_{11} - B_{21}A_{11} & A_{21}B_{12} - B_{21}A_{12} 
\end{bmatrix}.
\]
By the arbitrariness of $B$, we get that $A$ is of the form
$ \begin{bmatrix} 0 & 0 \\ 0 & A_{22} \end{bmatrix}$. This means that $A\in K\text{-}{\rm Span} \{ e_{ij} \mid i, j \in \hat{\mathcal{J}} \}$.
On the other hand, if $i,j\in \hat{\mathcal{J}}$, it is clear
that $e_{ij}\cdot B=B\cdot e_{ij}=0$ for all $B\in \mathfrak{M}_n(r)$. This
completes the proof of the lemma.
\end{proof}

\section{Commuting maps}\label{xxsec3}

In this section, we study commuting maps on the
inflated algebra $\mathfrak{M}_n(r)$. Let us first recall
a well-known result of commuting maps.
 
\begin{lemma}\label{l21}
Let $\mathcal{A}$ be a $K$-algebra with a $K$-basis $\Gamma$. Then a $K$-linear map $\theta: \mathcal{A} \to \mathcal{A}$ is a commuting map if and only if $[\theta(x), y] = [x, \theta(y)]$ for all $x, y\in \Gamma$.
\end{lemma}

Since $\{e_{xy} \mid x,y \in \mathbf{n} \}$ forms a $K$-basis of $\mathfrak{M}_n(r)$, we denote for all $i, j \in \mathbf{n}$,
\[
\theta(e_{ij}) = \sum_{x,y \in \mathbf{n}} C_{xy}^{ij} e_{xy},
\]
where $C_{xy}^{ij}\in K$ are the structure constants of $\theta$.

\begin{lemma}\label{l22}
The commuting map $\theta$ satisfies:
 \begin{align}
 \theta(e_{ii}) &= \sum_{x \in \mathcal{J}} C_{xx}^{ii} e_{xx} + \sum_{x,y \in \hat{\mathcal{J}}} C_{xy}^{ii} e_{xy}, \quad i   \in \mathbf{n},  \label{eq1}\\
 \theta(e_{ij}) &= \sum_{x \in \mathcal{J}} C_{xx}^{ij} e_{xx} + C_{ij}^{ij} e_{ij} + 
 \sum_{\substack{(x,y) \in \hat{\mathcal{J}}\times \hat{\mathcal{J}}\setminus \{(i,j)\}}} C_{xy}^{ij} e_{xy}, \quad i \neq j \in \mathbf{n}.\label{eq2}
 \end{align}
\end{lemma}
\begin{proof}
Without loss of generality, we assume that $|\mathbf{n}| \geq 2$ and $|\mathcal{J}| \geq 1$. In order to determine the coefficients $C_{xy}^{ii}$, we consider in two cases.

\textbf{Case 1.1.} $i \in \mathcal{J}$.
 Since $\theta$ is a commuting map,  we get $[\theta(e_{ii}), e_{ii}] = 0$. Let $k \in \mathbf{n} \setminus \{i\}$. Equating the coefficients of $e_{ik}$ and $e_{ki}$, we have
  \begin{align}\label{eq3}
  C_{ik}^{ii} = C_{ki}^{ii} = 0, \quad \text{for all } k \in \mathbf{n} \setminus \{i\}.
  \end{align}
Since $[\theta(e_{ii}), e_{xx}] = [e_{ii}, \theta(e_{xx})]$ for all $x \in \mathcal{J}$ by Lemma \ref{l21}, equating the coefficients of $e_{xy}$ and $e_{yx}$, with $y \neq x$ and $y \neq i$, then we have
 \begin{align}\label{eq4}
 C_{xy}^{ii} = C_{yx}^{ii} = 0, \quad \text{for all } x \in \mathcal{J}, y \in \mathbf{n} \setminus \{i,x\}.
 \end{align}

Combining the equations (\ref{eq3}) and (\ref{eq4}), we obtain
 \begin{align*}
 \theta(e_{ii})=& \sum_{x,y \in \mathcal{J}} C_{xy}^{ii} e_{xy}
 +\sum_{(x,y) \in \mathcal{J} \times \hat{\mathcal{J}}} (C_{xy}^{ii} e_{xy}+C_{yx}^{ii} e_{yx})+\sum_{x,y \in \hat{\mathcal{J}}} C_{xy}^{ii} e_{xy}\\
=&\sum_{x \in \mathcal{J}} C_{xx}^{ii} e_{xx} + \sum_{x,y \in \mathcal{J}, x\neq y} C_{xy}^{ii} e_{xy}
 +\sum_{y\in \hat{\mathcal{J}}} (C_{iy}^{ii} e_{iy}+C_{yi}^{ii} e_{yi})\\
  &+\sum_{x \in \mathcal{J}\setminus \{i\}, y\in \hat{\mathcal{J}}} (C_{xy}^{ii} e_{xy}+C_{yx}^{ii} e_{yx})+\sum_{x,y \in \hat{\mathcal{J}}} C_{xy}^{ii} e_{xy}\\
=& \sum_{x \in \mathcal{J}} C_{xx}^{ii} e_{xx} + \sum_{x, y \in \hat{\mathcal{J}}} C_{xy}^{ii} e_{xy}.
 \end{align*}
This proves (\ref{eq1}) for the case $ i \in \mathcal{J} $.

\textbf{Case 1.2.} $ i \in \hat{\mathcal{J}} $. Since $ [\theta(e_{ii}), e_{xx}] = [e_{ii}, \theta(e_{xx})] = e_{ii} \cdot  \theta(e_{xx}) - \theta(e_{xx}) \cdot  e_{ii} = 0 $ for all $ x \in \mathcal{J} $, we have 
\[
0 = \theta(e_{ii}) \cdot  e_{xx} - e_{xx} \cdot  \theta(e_{ii}) = \theta(e_{ii}) e_{xx} - e_{xx} \theta(e_{ii}).
\]
Then, equating the coefficients of $ e_{xy} $ and $ e_{yx} $, where $ y \in \mathbf{n} \setminus \{x\} $, we obtain 
 \begin{align}\label{eq5}
 C_{xy}^{ii} = C_{yx}^{ii} = 0, \quad \text{for all } x \in \mathcal{J},\, y \in \mathbf{n} \setminus \{x\}.
 \end{align}
It follows from the equation (\ref{eq5}) that
\[
\begin{aligned}
&\theta(e_{ii}) = \sum_{\substack{x,y \in \mathbf{n}}} C_{xy}^{ii} e_{xy} \\
=&\sum_{x,y \in \mathcal{J}} C_{xy}^{ii} e_{xy}
 +\sum_{(x,y) \in \mathcal{J} \times \hat{\mathcal{J}}} (C_{xy}^{ii} e_{xy}+C_{yx}^{ii} e_{yx})+\sum_{x,y \in \hat{\mathcal{J}}} C_{xy}^{ii} e_{xy}\\
=& \sum_{x \in \mathcal{J}} C_{xx}^{ii} e_{xx} + \sum_{x,y \in \mathcal{J}, x\neq y} C_{xy}^{ii} e_{xy}+
\sum_{x \in \mathcal{J}, y\in \hat{\mathcal{J}}} (C_{xy}^{ii} e_{xy}+C_{yx}^{ii} e_{yx})+\sum_{x,y \in \hat{\mathcal{J}}} C_{xy}^{ii} e_{xy}\\
=& \sum_{x \in \mathcal{J}} C_{xx}^{ii} e_{xx} + \sum_{\substack{x , y \in \hat{\mathcal{J}}}} C_{xy}^{ii} e_{xy}.
\end{aligned}
\]
This proves (\ref{eq1}) for the case $ i \in \hat{\mathcal{J}} $.

For all $ i,j \in \mathbf{n} $ with $ i \neq j $, we will prove that the coefficients of $\theta(e_{ij})$ satisfy equation (\ref{eq2}) by dividing into four cases.

\textbf{Case 2.1.} $ i,j \in \mathcal{J} $. Let $ x \in \mathcal{J} $. Since $ \theta $ is a commuting map, we have $ [\theta(e_{ij}), e_{xx}] = [e_{ij}, \theta(e_{xx})] $.
On the one hand, by (\ref{eq1}), 
\[
\begin{aligned} 
\relax[e_{ij},\theta(e_{xx})]&= e_{ij}\cdot \theta(e_{xx})-\theta(e_{xx})\cdot  e_{ij} = e_{ij} ( \sum_{y \in \mathcal{J}} C_{yy}^{xx} e_{yy})-(\sum_{y \in \mathcal{J}} C_{yy}^{xx} e_{yy}) e_{ij} \\
&= ( C_{jj}^{xx} - C_{ii}^{xx} ) e_{ij}.
\end{aligned}
\]
On the other hand, 
$
[\theta(e_{ij}), e_{xx}] = \theta(e_{ij}) \cdot  e_{xx} - e_{xx} \cdot  \theta(e_{ij}) = \theta(e_{ij}) e_{xx} - e_{xx} \theta(e_{ij}).
$
Combining the just above two identities, we get  
 \begin{align}\label{eq6}
 \theta(e_{ij}) e_{xx} - e_{xx} \theta(e_{ij}) = \left( C_{jj}^{xx} - C_{ii}^{xx} \right) e_{ij},   \text{ for all }  i, j,x \in \mathcal{J}.
 \end{align}
If $ x = i $ and $ y \in \mathbf{n} \setminus \{j\} $, then, by equating the coefficients of $ e_{iy} $ and $ e_{ij} $, we have 
 \begin{align}\label{eq7}
 C_{iy}^{ij} = 0, \quad \text{for all } i \neq j \in \mathcal{J} \text{ and } y \in \mathbf{n} \setminus \{j\}.
 \end{align}
and 
 \begin{align}\label{eq8}
C_{ij}^{ij} = C_{ii}^{ii} - C_{jj}^{ii}, \quad \text{for all } i \neq j \in \mathcal{J}. 
\end{align}
If $ x = j $ in (\ref{eq6}), for all $ y \in \mathbf{n} \setminus \{i\} $, then  by equating the coefficients of $ e_{yj} $ and $ e_{ij} $, we have 
 \begin{align}\label{eq9}
C_{yj}^{ij} = 0, \quad \text{for all } i \neq j \in \mathcal{J} \text{ and } y \in \mathbf{n} \setminus \{i\}. 
\end{align}
and 
 \begin{align}\label{eq10}
C_{ij}^{ij} = C_{jj}^{jj} - C_{ii}^{jj}, \quad \text{for all }  i \neq j \in \mathcal{J}. 
 \end{align}
If $ x \neq i $ in (\ref{eq6}), then $ e_{xx} \theta(e_{ij}) (I_n - e_{xx}) = 0 $, we have
 \begin{align}\label{eq11}
C_{xy}^{ij} = 0, \quad x \in \mathcal{J} \setminus \{i\}, \, y \in \mathbf{n} \setminus \{x\}.
\end{align}
If $ x \neq j $ in (\ref{eq6}), then $ (I_n - e_{xx}) \theta(e_{ij}) e_{xx} = 0 $, we have
 \begin{align}\label{eq12}
C_{yx}^{ij} = 0, \quad x \in \mathcal{J} \setminus \{j\}, \, y \in \mathbf{n} \setminus \{x\}.
\end{align}
Combining the equations (\ref{eq7}), (\ref{eq9}), (\ref{eq11}) and (\ref{eq12}), we can express $\theta(e_{ij})$ as follows
\[
\begin{aligned}
&\theta(e_{ij})=\sum_{\substack{x,y \in \mathbf{n}}} C_{xy}^{ij} e_{xy}\\
=& \sum_{x,y \in \mathcal{J}} C_{xy}^{ij} e_{xy} + \sum_{\substack{x \in \mathcal{J} \\ y \in \hat{\mathcal{J}}}} C_{xy}^{ij} e_{xy} + \sum_{\substack{x \in \hat{\mathcal{J}} \\ y \in \mathcal{J}}} C_{xy}^{ij} e_{xy} + \sum_{\substack{x,y \in \hat{\mathcal{J}}}} C_{xy}^{ij} e_{xy} \\
=& \sum_{x \in \mathcal{J}} C_{xx}^{ij} e_{xx} +\sum_{\substack{x \in \mathcal{J}\setminus \{i\} \\y \in \mathcal{J}\setminus \{j\}, x\neq y }} C_{xy}^{ij} e_{xy} + \sum_{\substack{x \in \mathcal{J}\setminus \{i\} \\y \in \hat{\mathcal{J}} }} C_{xy}^{ij} e_{xy} +
\sum_{\substack{x \in \hat{\mathcal{J}}\\y \in \mathcal{J}\setminus \{j\} }} C_{xy}^{ij} e_{xy} \\
&+\sum_{\substack{x \in \mathcal{J}\setminus \{i\}  }} C_{xj}^{ij} e_{xj}
\sum_{\substack{y \in \mathcal{J}\setminus \{j\} }} C_{iy}^{ij} e_{iy}+C_{ij}^{ij} e_{ij}
+ \sum_{\substack{x,y \in \hat{\mathcal{J}}}} C_{xy}^{ij} e_{xy} \\
=& \sum_{x \in \mathcal{J}} C_{xx}^{ij} e_{xx} + C_{ij}^{ij} e_{ij} + 
 \sum_{\substack{(x,y) \in \hat{\mathcal{J}}\times \hat{\mathcal{J}}}} C_{xy}^{ij} e_{xy}.
\end{aligned}
\]
Thus we prove (\ref{eq2}) for the case $ i \neq j \in \mathcal{J} $.

\textbf{Case 2.2.} $ i \in \mathcal{J},~j \in \hat{\mathcal{J}} $. Let $ x \in \mathcal{J} $. Since $ \theta $ is a commuting map, $ [\theta(e_{ij}), e_{xx}] = [e_{ij}, \theta(e_{xx})] $.
On the one hand, by (\ref{eq1}),
$
[e_{ij}, \theta(e_{xx})] = e_{ij} \cdot  \theta(e_{xx}) - \theta(e_{xx}) \cdot  e_{ij} = -C_{ii}^{xx} e_{ij}.
$
On the other hand, $ [\theta(e_{ij}), e_{xx}] = \theta(e_{ij}) e_{xx} - e_{xx} \theta(e_{ij}) $.
It follows that
 \begin{align}\label{eq13}
\theta(e_{ij}) e_{xx} - e_{xx} \theta(e_{ij}) = -C_{ii}^{xx} e_{ij} .
 \end{align}
If $ x = i $ in  (\ref{eq13}), then  by equating the coefficients of $ e_{iy} $ and $ e_{ij} $, we have
 \begin{align}\label{eq14}
C_{iy}^{ij} = 0, \quad \text{for all } y \in \mathbf{n} \setminus \{i, j\},
 \end{align}
and
 \begin{align}\label{eq15}
C_{ij}^{ij} = C_{ii}^{ii},\quad \text{for all } i \in \mathcal{J}, j \in \hat{\mathcal{J}}. 
 \end{align}
If $ x \neq i $ in (\ref{eq13}), then  $ e_{xx} \theta(e_{ij}) (I_n - e_{xx}) = 0 $ and hence
 \begin{align}\label{eq16}
C_{xy}^{ij} = 0, \quad \text{for all } x \in \mathcal{J} \setminus \{i\}, y \in \mathbf{n} \setminus \{x\}.
 \end{align}
Since $ x \in \mathcal{J}$ and $ j \in \hat{\mathcal{J}} $ imply  $ x \neq j $, we have $ (I_n - e_{xx}) \theta(e_{ij}) e_{xx} = 0 $ for all $x\in \mathcal{J}$ by (\ref{eq13}). This leads to
 \begin{align}\label{eq17}
C_{yx}^{ij} = 0, \quad \text{for all } x \in \mathcal{J}, y \in \mathbf{n} \setminus \{x\}.
 \end{align}
Now combining the equations (\ref{eq14}), (\ref{eq16}) and (\ref{eq17}), we obtain
\[
\begin{aligned}
\theta(e_{ij}) 
=& \sum_{\substack{x,y \in \mathbf{n}}} C_{xy}^{ij} e_{xy}=\sum_{\substack{x \in \mathcal{J} \\ y \in \mathbf{n}}} C_{xy}^{ij} e_{xy} + \sum_{\substack{x \in \hat{\mathcal{J}} \\ y \in \mathcal{J}}} C_{xy}^{ij} e_{xy} + \sum_{\substack{x,y \in \hat{\mathcal{J}}}} C_{xy}^{ij} e_{xy} \\
=& \sum_{x \in \mathcal{J}} C_{xx}^{ij} e_{xx} + \sum_{\substack{y \in \mathbf{n}\setminus \{i,j\} }} C_{iy}^{ij} e_{iy}+C_{ij}^{ij} e_{ij}
+\sum_{\substack{x \in \mathcal{J}\setminus \{i\} \\y \in \mathbf{n}\setminus \{x\} }} C_{xy}^{ij} e_{xy} \\
&+\sum_{\substack{x \in \hat{\mathcal{J}}\\y \in \mathcal{J} }} C_{xy}^{ij} e_{xy}+ \sum_{\substack{x,y \in \hat{\mathcal{J}}}} C_{xy}^{ij} e_{xy} \\
=& \sum_{x \in \mathcal{J}} C_{xx}^{ij} e_{xx} + C_{ij}^{ij} e_{ij} + 
 \sum_{\substack{(x,y) \in \hat{\mathcal{J}}\times \hat{\mathcal{J}}}} C_{xy}^{ij} e_{xy}.
\end{aligned}
\]
This proves (\ref{eq2}) for the case $i \in \mathcal{J}, ~j \in \hat{\mathcal{J}}$.

\textbf{Case 2.3.} $i \in \hat{\mathcal{J}},~j \in \mathcal{J}$. Let $x \in \mathcal{J}$. Since $\theta$ is a commuting map, we have $[\theta(e_{ij}), e_{xx}] = [e_{ij}, \theta(e_{xx})]$.
On the one hand, by (\ref{eq1}) and the multiplication of $\mathfrak{M}_n(r)$, $[e_{ij}, \theta(e_{xx})] = e_{ij} \cdot  \theta(e_{xx}) - \theta(e_{xx}) \cdot  e_{ij} = C_{jj}^{xx} e_{ij}$.
On the other hand, $[\theta(e_{ij}), e_{xx}] = \theta(e_{ij}) e_{xx} - e_{xx} \theta(e_{ij})$. It follows that
 \begin{align}\label{eq18}
\theta(e_{ij}) e_{xx} - e_{xx} \theta(e_{ij}) = C_{jj}^{xx} e_{ij}.
 \end{align}
Since $x \in \mathcal{J},~i \in \hat{\mathcal{J}}$ imply $x \neq i$, we obtain $e_{xx} \theta(e_{ij}) (I_n - e_{xx}) = 0$. This leads to
 \begin{align}\label{eq19}
C_{xy}^{ij} = 0, \quad \text{for all } x \in \mathcal{J}, ~y \in \mathbf{n} \setminus \{x\}.  
\end{align}
If $x = j$ in (\ref{eq18}). Then, by equating the coefficients of $e_{yj}$ and $e_{ij}$, $y \in \mathbf{n} \setminus \{i, j\}$, we have
 \begin{align}\label{eq20}
C_{yj}^{ij} = 0, \quad \text{for all } y \in \mathbf{n} \setminus \{i, j\},
 \end{align}
and
 \begin{align}\label{eq21}
C_{ij}^{ij} = C_{jj}^{jj}, \quad \text{for all } i \in \hat{\mathcal{J}}, ~j \in \mathcal{J}. 
 \end{align}
If $x \neq j$ in (\ref{eq18}). Then $(I_n - e_{xx}) \theta(e_{ij}) e_{xx} = 0$, we have
 \begin{align}\label{eq22}
C_{yx}^{ij} = 0, \quad \text{for all } x \in \mathcal{J} \setminus \{j\},~ y \in \mathbf{n} \setminus \{x\} .
\end{align}
Combining the equations (\ref{eq19}), (\ref{eq20}) and (\ref{eq22}), we obtain
\[
\begin{aligned}
\theta(e_{ij})
=& \sum_{\substack{x,y \in \mathbf{n}}} C_{xy}^{ij} e_{xy}
= \sum_{\substack{x \in \mathbf{n} \\ y \in \mathcal{J}}} C_{xy}^{ij} e_{xy}+ \sum_{\substack{x \in \mathcal{J} \\ y \in \hat{\mathcal{J}}}} C_{xy}^{ij} e_{xy} +\sum_{\substack{x,y \in \hat{\mathcal{J}}}} C_{xy}^{ij} e_{xy}\\
=&\sum_{x \in \mathcal{J}} C_{xx}^{ij} e_{xx} + \sum_{\substack{x \in \mathbf{n}\setminus \{i,j\} }} C_{xj}^{ij} e_{xj}+C_{ij}^{ij} e_{ij}+\sum_{\substack{x \in \mathcal{J} \\y \in \mathbf{n}\setminus \{j, x\} }} C_{xy}^{ij} e_{xy} \\
&+\sum_{\substack{x \in \mathcal{J}\\y \in \hat{\mathcal{J}} }} C_{xy}^{ij} e_{xy}+ \sum_{\substack{x,y \in \hat{\mathcal{J}}}} C_{xy}^{ij} e_{xy}  \\
=& \sum_{x \in \mathcal{J}} C_{xx}^{ij} e_{xx} + C_{ij}^{ij} e_{ij} + 
 \sum_{\substack{(x,y) \in \hat{\mathcal{J}}\times \hat{\mathcal{J}}}} C_{xy}^{ij} e_{xy}.
\end{aligned}
\]
This proves (\ref{eq2}) for the case $i \in \hat{\mathcal{J}}, ~j \in \mathcal{J}$.

\textbf{Case 2.4.} $i, j \in \hat{\mathcal{J}}$ and $i \neq j$. Let $x \in \mathcal{J}$. Since $\theta$ is a commuting map, we have $[\theta(e_{ij}), e_{xx}] = [e_{ij}, \theta(e_{xx})] = 0$ and therefore $\theta(e_{ij}) e_{xx} - e_{xx} \theta(e_{ij}) = 0$.
By equating the coefficients of $e_{yx}$ and $e_{xy}$, where $y \in \mathbf{n} \setminus \{x\}$, we have
 \begin{align}\label{eq23}
C_{xy}^{ij} = C_{yx}^{ij} = 0, \quad \text{for all } x \in \mathcal{J},~y \in \mathbf{n} \setminus \{x\}.
\end{align}
It follows from the equation (\ref{eq23}) that
\[
\begin{aligned}
\theta(e_{ij}) =& \sum_{\substack{x,y \in \mathbf{n}}} C_{xy}^{ij} e_{xy} \\
=& \sum_{x \in \mathcal{J}} C_{xx}^{ij} e_{xx} + \sum_{\substack{x \in \hat{\mathcal{J}} \\ y \in \mathcal{J}}} C_{xy}^{ij} e_{xy} + \sum_{\substack{x \in \mathcal{J} \\ y \in \mathbf{n} \setminus \{x\}}} C_{xy}^{ij} e_{xy} + \sum_{\substack{x,y \in \hat{\mathcal{J}}}} C_{xy}^{ij} e_{xy} \\
=& \sum_{x \in \mathcal{J}} C_{xx}^{ij} e_{xx}  + \sum_{\substack{(x,y) \in \hat{\mathcal{J}}\times \hat{\mathcal{J}}}} C_{xy}^{ij} e_{xy}\\
=& \sum_{x \in \mathcal{J}} C_{xx}^{ij} e_{xx} + C_{ij}^{ij} e_{ij} + \sum_{\substack{(x,y) \in \hat{\mathcal{J}}\times \hat{\mathcal{J}}\setminus \{(i,j)\}}} C_{xy}^{ij} e_{xy}.
\end{aligned}
\]
This proves (\ref{eq2}) for the case $i\neq j \in \hat{\mathcal{J}}$ and we complete the proof of this lemma.
\end{proof}

\begin{lemma}\label{l23}
Let $i \in \mathbf{n}$. Then $C_{kk}^{ii} = C_{ll}^{ii}$, for all $k, l \in \mathcal{J} \setminus \{i\}$. Moreover, if $\hat{\mathcal{J}} \neq \emptyset$, then $C_ {kk}^{ii} = 0$ for all $k \in \mathcal{J} \setminus \{i\}$.
\end{lemma}

\begin{proof}
Let $k, l \in \mathbf{n} \setminus \{i\}$ with $k \neq l$.
Since $\theta$ is a commuting map, we have $[\theta(e_{ii}), e_{kl}] = [e_{ii}, \theta(e_{kl})]$. Then the formulas (\ref{eq1}) and (\ref{eq2}) imply that
\[
\begin{aligned}
&\sum_{x \in \mathcal{J}} C_{xx}^{ii} e_{xx} \cdot  e_{kl} - e_{kl}\cdot  \sum_{x \in \mathcal{J}} C_{xx}^{ii} e_{xx}\\
 =& e_{ii} \cdot  \left( \sum_{x \in \mathcal{J}} C_{xx}^{kl} e_{xx} + C_{kl}^{kl} e_{kl} \right) - \left( \sum_{x \in \mathcal{J}} C_{xx}^{kl} e_{xx} + C_{kl}^{kl} e_{kl} \right) \cdot  e_{ii}
\end{aligned}
\]
and hence
 \begin{align}\label{eq24}
\sum_{x \in \mathcal{J}} C_{xx}^{ii} ( \delta_{xk}-\delta_{xl}) e_{kl} = C_{kl}^{kl}e_{ii} \cdot  e_{kl} - C_{kl}^{kl} e_{kl} \cdot  e_{ii}. 
\end{align}

If $i, k, l \in \mathcal{J}$, the equation (\ref{eq24}) implies that
 \begin{align}\label{eq25}
(C_{kk}^{ii} - C_{ll}^{ii}) e_{kl} = (\delta_{ik} - \delta_{il}) C_{kl}^{kl} e_{kl}. 
 \end{align}
Noting that $k \neq l$, we can get the relation $C_{kk}^{ii} = C_{ll}^{ii}$ from (\ref{eq25}) by setting $k \neq i \neq l$ in $\mathcal{J}$.
This  proves the first statement of the lemma when  $\hat{\mathcal{J}}= \emptyset$.
Next, we assume without loss of generality that $\hat{\mathcal{J}} \neq \emptyset$.

If $i = l \in \hat{\mathcal{J}}$ and $k \in \mathcal{J}$ in (\ref{eq24}), we have $\delta_{xl}=0$ for all $x\in \mathcal{J}$. This implies that $\left( \sum_{x \in \mathcal{J}} C_{xx}^{ii} e_{xx} \right) e_{ki} = 0$ and hence
 \begin{align}\label{eq26}
C_{kk}^{ii} = 0, \quad \text{for all } i \in \hat{\mathcal{J}},~k \in \mathcal{J} .
\end{align}
Thus we have $C_{kk}^{ii} = C_{ll}^{ii}$ when $i \in \hat{\mathcal{J}}$ and $k, l \in \mathcal{J}$. We have
completed the proof of the first statement.

If $i \neq l \in \mathcal{J}$, $k \in \hat{\mathcal{J}}$ in (\ref{eq24}), we have $\delta_{xk}=0$ for all $x\in \mathcal{J}$, which in turn shows
 \begin{align}\label{eq27}
C_{ll}^{ii} = 0, \quad \text{for all } i \neq l \in \mathcal{J} \text{ and } \hat{\mathcal{J}} \neq \emptyset. 
 \end{align}
Combining (\ref{eq26}) and (\ref{eq27}), we have $C_{kk}^{ii} = 0$ for all $k \in \mathcal{J} \setminus \{i\}$ if $\hat{\mathcal{J}} \neq \emptyset$.
\end{proof}

By Lemma \ref{l23} we can rewrite the identity (\ref{eq1})
as follows.

\begin{corollary}\label{l24}
For any $i \in \mathbf{n}$, there is
\[
\begin{aligned}
\theta(e_{ii}) =& \lambda_i\sum_{x \in \mathcal{J}}  e_{xx} + (C_{ii}^{ii} - \lambda_i) e_{ii}, \quad \text{if }\ \hat{\mathcal{J}} = \emptyset;\\
\theta(e_{ii}) =& C_{ii}^{ii} e_{ii} + \sum_{(x, y) \in \hat{\mathcal{J}}\times \hat{\mathcal{J}}\setminus \{(i,i)\}} C_{xy}^{ii} e_{xy},
\quad \text{if }\ \hat{\mathcal{J}} \neq \emptyset,
\end{aligned}
\]
where $\lambda_i=C_{kk}^{ii}$ for all $k\in \mathcal{J}\setminus  \{i\}$.
\end{corollary}

\begin{lemma}\label{l25}
 Let $i \neq j\in\mathbf{n}$. Then $C_{kk}^{ij} = C_{ll}^{ij}$ for all $k, l \in \mathcal{J}$.
Moreover, if $\hat{\mathcal{J}} \neq \emptyset$, then $C_{kk}^{ij} = 0$ for all $k \in \mathcal{J}$.
\end{lemma}
\begin{proof}
Let $i \neq j$ and $k \neq l$ in $\mathbf{n}$. Since $\theta$ is a commuting map,   we have
\[
[\theta(e_{ij}), e_{kl}] = [e_{ij}, \theta(e_{kl})].
\]
It follows from the formula (\ref{eq2}) that
\[
\begin{aligned}
&\left( C_{ij}^{ij} e_{ij} + \sum_{x \in \mathcal{J}} C_{xx}^{ij} e_{xx} \right) \cdot  e_{kl} - e_{kl} \cdot  \left( C_{ij}^{ij} e_{ij} + \sum_{x \in \mathcal{J}} C_{xx}^{ij} e_{xx} \right) \\
&= e_{ij} \cdot  \left( \sum_{x \in \mathcal{J}} C_{xx}^{kl} e_{xx} + C_{kl}^{kl} e_{kl} \right) - \left( \sum_{x \in \mathcal{J}} C_{xx}^{kl} e_{xx} + C_{kl}^{kl} e_{kl} \right) \cdot  e_{ij},
\end{aligned}
\]
which in turn shows that
 \begin{align}\label{eq28}
\left( C_{ij}^{ij} - C_{kl}^{kl} \right) [e_{ij}, e_{kl}] + \sum_{x \in \mathcal{J}} C_{xx}^{ij} [e_{xx}, e_{kl}] = \sum_{x \in \mathcal{J}} C_{xx}^{kl} [e_{ij}, e_{xx}]. 
 \end{align}

If $i, j, k, l \in \mathcal{J}$, it follows from (\ref{eq28}) that
\[
\left( C_{kk}^{ij} - C_{ll}^{ij} \right) e_{kl} + \left( C_{ij}^{ij} - C_{kl}^{kl} \right) \left( \delta_{jk} e_{il} - \delta_{il} e_{kj} \right) = \left( C_{jj}^{kl} - C_{ii}^{kl} \right) e_{ij}.
\]
When $k, l \in \mathcal{J} \setminus \{i\}$, equating the coefficient of $e_{kl}$, we have
 \begin{align}\label{eq29}
C_{kk}^{ij} = C_{ll}^{ij}, \quad \text{if } k \neq i \neq l \text{ in } \mathcal{J}.
 \end{align}
When $k \in \mathcal{J} \setminus \{i\}$ and $l = i$, equating the coefficient of $e_{ki}$, we have
 \begin{align}\label{eq30}
C_{kk}^{ij} = C_{ii}^{ij}, \quad \text{if } k \neq i \text{ in } \mathcal{J}.
 \end{align}
Then equations (\ref{eq29}) and (\ref{eq30}) imply that
 \begin{align}\label{eq31}
C_{kk}^{ij} = C_{ll}^{ij}, \quad \text{for all } k, l \in \mathcal{J}.
 \end{align}
The equation (\ref{eq31}) proves the first statement of the lemma when  $\hat{\mathcal{J}}= \emptyset$.

Next, we assume without loss of generality that $\hat{\mathcal{J}} \neq \emptyset$.
To study the coefficients $C_{kk}^{ij}$ for all $k\in \mathcal{J}$, there are four cases occurring.

\textbf{Case 1.} $i, j \in \mathcal{J}$. Taking $l\in \mathcal{J}$ and $k \in \hat{\mathcal{J}}$ in (\ref{eq28}),
we have
\[
\left( C_{kl}^{kl} - C_{ij}^{ij} \right) \delta_{il} e_{kj} - C_{ll}^{ij} e_{kl} = \left( C_{jj}^{kl} - C_{ii}^{kl} \right) e_{ij}.
\]
Since $i \in \mathcal{J}$ and $k \in \hat{\mathcal{J}}$, there is $k \neq i$.
If $l = i$, by equating the coefficients of $e_{ki}$, we can get
 \begin{align}\label{eq32}
C_{ii}^{ij} = 0, \quad \text{for all } i, j \in \mathcal{J}.
\end{align}
If $l \neq i$, by equating the coefficients of $e_{kl}$,
we can get
 \begin{align}\label{eq33}
C_{ll}^{ij} = 0, \quad \text{for all } i, j \in \mathcal{J},~ l \in \mathcal{J} \setminus \{i\}.
 \end{align}
Combining the identities (\ref{eq32}) and (\ref{eq33}),
we have
 \begin{align}\label{eq34}
C_{ll}^{ij} = 0, \quad \text{for all } i, j, l \in \mathcal{J}.
 \end{align}

\textbf{Case 2.} $i \in \mathcal{J}$, $j \in \hat{\mathcal{J}}$. For any $k, l \in \mathbf{n}$, it follows from (\ref{eq28}) that
 \begin{align}\label{eq35}
\left( -C_{ij}^{ij} + C_{kl}^{kl} \right) e_{kl} \cdot  e_{ij} + \sum_{x \in \mathcal{J}} C_{xx}^{ij} [e_{xx}, e_{kl}] = -C_{ii}^{kl} e_{ij}. 
\end{align}
If $k \neq i \in \mathcal{J}$ and $l = j \in \hat{\mathcal{J}}$, then $e_{kl} \cdot  e_{ij} = 0$. By equating the coefficients of $e_{kj}$, we have
 \begin{align}\label{eq36}
C_{kk}^{ij} = 0, \quad \text{for all } k \neq i \in \mathcal{J}, j \in \hat{\mathcal{J}}.
 \end{align}
If $l = i \in \mathcal{J}$ and $k = j \in \hat{\mathcal{J}}$, then by (\ref{eq35}), we can get
\[
\left( C_{ji}^{ji} - C_{ij}^{ij} \right) e_{jj} - C_{ii}^{ij} e_{ji} = -C_{ii}^{ji} e_{ij}.
\]
By equating the coefficients of $e_{ji}$ and $e_{ij}$, we get
 \begin{align}\label{eq37}
C_{ii}^{ij} = 0, \quad \text{for all } i \in \mathcal{J} \text{ and } j \in \hat{\mathcal{J}},
\end{align}
and
 \begin{align}\label{eq38}
C_{ii}^{ji} = 0, \quad \text{for all } i \in \mathcal{J} \text{ and } j \in \hat{\mathcal{J}}.
 \end{align}
Combining (\ref{eq36}) and (\ref{eq37}), we have
 \begin{align}\label{eq39}
C_{kk}^{ij} = 0, \quad \text{for all } i, k \in \mathcal{J}, j \in \hat{\mathcal{J}}.
 \end{align}

\textbf{Case 3.} $i \in \hat{\mathcal{J}}$, $j \in \mathcal{J}$. Taking 
$l \neq j$ and $k=i$ in (\ref{eq28}), we have
by equating the coefficients of $e_{il}$ that
 \begin{align}\label{eq40}
C_{ll}^{ij} = 0, \quad \text{for all } i \in \hat{\mathcal{J}} \text{ and } l \neq j \in \mathcal{J}.
 \end{align}
Notice that we can rewrite (\ref{eq38}) as $C_{jj}^{ij}=0$ for $i \in \hat{\mathcal{J}}$, $j \in \mathcal{J}$.
This fact together with (\ref{eq40}) shows that
 \begin{align}\label{eq41}
C_{ll}^{ij} = 0, \quad \text{for all } i \in \hat{\mathcal{J}},~ l, j \in \mathcal{J}.
 \end{align}

\textbf{Case 4.} $i, j \in \hat{\mathcal{J}}$. Taking $l=j$ and $k \in \mathcal{J}$ in (\ref{eq28}), we have
by equating the coefficients of $e_{kj}$ that
 \begin{align}\label{eq42}
C_{kk}^{ij} = 0, \quad \text{for all } i, j \in \hat{\mathcal{J}}, k \in \mathcal{J}.
 \end{align}
Combining (\ref{eq34}), (\ref{eq39}), (\ref{eq41}) and (\ref{eq42}), we obtain $C_{kk}^{ij} = 0$ for all $k \in \mathcal{J}$ and $i\neq j \in \mathbf{n}$, when $\hat{\mathcal{J}} \neq \emptyset$. This fact together with (\ref{eq31}) completes the proof of the lemma.
\end{proof}

By Lemma \ref{l25} we can rewrite the identity (\ref{eq2})
as follows.

\begin{corollary}\label{l26}
For any $i \neq j\in \mathbf{n}$, there is 
\[
\begin{aligned}
\theta(e_{ij}) =& \lambda_{ij} \sum_{x \in \mathcal{J}} e_{xx} + C_{ij}^{ij} e_{ij},
\quad if\ \hat{\mathcal{J}} = \emptyset;\\
\theta(e_{ij}) =& C_{ij}^{ij} e_{ij} + \sum_{(x, y) \in \hat{\mathcal{J}}\times \hat{\mathcal{J}}\setminus \{(i,j)\}} C_{xy}^{ij} e_{xy},
\quad if\ \hat{\mathcal{J}} \neq \emptyset,
\end{aligned}
\]
where $\lambda_{ij}=C_{kk}^{ij}$ for all $k\in \mathcal{J}$.
\end{corollary}

\begin{lemma}\label{l27}
(i) Let $i,j,k,l\in \mathcal{J} $. If $ i \neq j $ and $ k \neq l $, then $C_{ij}^{ij} = C_{kl}^{kl} $.

(ii) If $ \hat{\mathcal{J}} \neq \emptyset $, then
$ C_{ij}^{ij} = C_{kl}^{kl} $, for all $ (i,j), (k,l) \in \mathbf{n} \times \mathbf{n} \setminus \hat{\mathcal{J}} \times \hat{\mathcal{J}} $.
\end{lemma}

\begin{proof}
(i) Let $i,j,k,l\in \mathcal{J}$ with $i\neq j $, $k\neq l$.
We have from the equations (\ref{eq8}) and (\ref{eq10})
that $C_{kl}^{kl} = C_{kk}^{kk} - C_{ll}^{kk}$ and  $C_{ij}^{ij} = C_{jj}^{jj} - C_{ii}^{jj}$.
If $j = k$, then $C_{ii}^{jj}=C_{ll}^{jj}$ by Lemma \ref{l23}, and hence
 \begin{align}\label{eq43}
C_{ij}^{ij} = C_{jj}^{jj} - C_{ii}^{jj} = C_{jj}^{jj} - C_{ll}^{jj} = C_{jl}^{jl}, 
\quad \text{for all } i \neq j   \neq l \text{ in } \mathcal{J}.
 \end{align}
If $j \neq k$, applying equation (\ref {eq43}) twice, we obtain
\[
C_{ij}^{ij} = C_{jk}^{jk} = C_{kl}^{kl},
\]
for all $ i \neq j $, $ j \neq k $ and $ k \neq l $ in $\mathcal{J}$. We complete the proof of the first claim.

(ii) Since $\hat{\mathcal{J}} \neq \emptyset $, by Lemma \ref{l23} we have $ C_{kk}^{ii} = 0 $ for all
$k \in \mathcal{J} \setminus \{i\}$. It follows from the
equations (\ref{eq8}) and (\ref{eq10}) that
$C_{ij}^{ij} = C_{jj}^{jj} = C_{ii}^{ii}$, for all $i\neq j$ in $\mathcal{J}$.
Then, by the just proved part (i), we have
\[
C_{jj}^{jj} = C_{ii}^{ii} =C_{ij}^{ij} = C_{kl}^{kl}=C_{kk}^{kk}=C_{ll}^{ll}, \quad \text{for all } i \neq j \text{ and } k \neq l \text{ in } \mathcal{J}.
\]
Consequently,
\begin{align}\label{eq44}
C_{ij}^{ij} = C_{kl}^{kl}, \quad \text{for all }  i,j,k,l \in \mathcal{J}.
\end{align}
If $ (i,j) \in \hat{\mathcal{J}} \times \mathcal{J} $ (resp. $ (i,j) \in \mathcal{J} \times \hat{\mathcal{J}} $), then we have $ C_{ij}^{ij} = C_{jj}^{jj}$ by (\ref{eq21})
(resp. $ C_{ij}^{ij} = C_{ii}^{ii} $ by (\ref{eq15})).
Using (\ref{eq44}) we can obtain that $ C_{ij}^{ij} = C_{kl}^{kl} $ for all $ (i,j), (k,l)\in \mathbf{n} \times \mathbf{n} \setminus \hat{\mathcal{J}} \times \hat{\mathcal{J}}$
 and complete the proof of the second claim.
\end{proof}

\begin{lemma}\label{l28}
If $ \hat{\mathcal{J}} \neq \emptyset $, then there exists $ \lambda \in K $ such that
$$ \theta(e_{ij}) = \lambda e_{ij} + \sum_{\substack{x,y \in \hat{\mathcal{J}}}} C_{xy}^{ij} e_{xy},
\quad \text{for all }  (i,j)  \in \mathbf{n} \times \mathbf{n} \setminus \hat{\mathcal{J}} \times \hat{\mathcal{J}}.$$
\end{lemma}

\begin{proof}
When $ \hat{\mathcal{J}} \neq \emptyset $, by Corollaries \ref{l24} and \ref{l26} we have
\[
\theta(e_{ij}) = C_{ij}^{ij} e_{ij} + \sum_{\substack{(x,y) \in \hat{\mathcal{J}}\times \hat{\mathcal{J}} \setminus 
\{(i,j)\} }} C_{xy}^{ij} e_{xy}, \quad \text{for all } ( i,j) \in \mathbf{n} \times \mathbf{n}.
\]
Especially,
 \begin{align}\label{eq45}
\theta(e_{ij}) = C_{ij}^{ij} e_{ij} + \sum_{\substack{x,y \in \hat{\mathcal{J}}}} C_{xy}^{ij} e_{xy}, \quad \text{for all } ( i,j) \in \mathbf{n} \times \mathbf{n} \setminus \hat{\mathcal{J}} \times \hat{\mathcal{J}}. 
 \end{align}
Notice that, by Lemma \ref{l27} (ii), the coefficients
$C_{ij}^{ij}$ are independent of the indices $(i,j) \in \mathbf{n} \times \mathbf{n} \setminus \hat{\mathcal{J}} \times \hat{\mathcal{J}}$. Hence we can set
$\lambda:=C_{ij}^{ij}$ in (\ref{eq45}) and complete the proof of the lemma.
\end{proof}

The main result of this paper is as follows.
\begin{theorem}\label{l29}
Let $\theta$ be a commuting map of the inflated algebra $\mathfrak{M}_n(r)$. Then $\theta$ is standard. Moreover,
every commuting map of $\mathfrak{M}_n(r)$ is proper if and only if $\hat{\mathcal{J}} = \emptyset$, i.e., $\mathfrak{M}_n(r)$ is the full matrix algebra ${\rm M}_n(K)$.
\end{theorem}

\begin{proof}
We have seen that the action of commuting map $\theta$
differ significantly at the two cases
$\hat{\mathcal{J}} = \emptyset$ and $\hat{\mathcal{J}} \neq \emptyset$. When $\hat{\mathcal{J}} = \emptyset$,
$\mathfrak{M}_n(r)$ is the full matrix algebra ${\rm M}_n(K)$,
and the desired result is well-known. However, we tend to give a self-contained proof here in the sense of inflated algebras.

\textbf{Case 1. $\hat{\mathcal{J}} = \emptyset$.}
We assume $|\mathbf{n}| = |\mathcal{J}| \geq 2$ without loss
of generality.
Let us define a $K$-linear map $\mathcal{L}: \mathfrak{M}_n(r) \to \mathfrak{M}_n(r)$ by $\mathcal{L}(e_{ij}) = C_{ij}^{ij} e_{ij}$ and $\mathcal{L}(e_{ii}) = C_{ij}^{ij} e_{ii}$, for all $i \neq j\in\mathbf{n}$.
Notice that $C_{ij}^{ij} = C_{kl}^{kl}$ for any $i \neq j$ and $k\neq l$ by Lemma \ref{l27} (i). Hence the map $\mathcal{L}$ is well-defined and is of the form $\mathcal{L}(f) = \lambda f$, for all $f \in \mathfrak{M}_n(r)$, where
$\lambda := C_{ij}^{ij}$ for all
$i \neq j$ is identified with the scalar matrix $\lambda I_n$.

Let $\mu := \theta - \mathcal{L}$. For any $j \neq i$,
it follows from Lemma \ref{l23} and Corollary \ref{l24} that
$\mu(e_{ii}) = C_{jj}^{ii} \sum_{\substack{x \in \mathcal{J}}} e_{xx} + \left( C_{ii}^{ii} - C_{jj}^{ii} - C_{ij}^{ij} \right) e_{ii}$. Combining this fact with the relation (\ref{eq8}), we have $\mu(e_{ii}) = C_{jj}^{ii} \sum_{\substack{x \in \mathcal{J} }} e_{xx} \in Z(\mathfrak{M}_n(r))$ by Lemma \ref{l1} (i). At the same time,
by Corollary \ref{l26} and the definition of $\lambda$,
we have $\mu(e_{ij}) = \lambda_{ij} \sum_{x \in \mathcal{J}} e_{xx} + C_{ij}^{ij} e_{ij} - \lambda e_{ij} = \lambda_{ij} \sum_{x \in \mathcal{J}} e_{xx}$, where $\lambda_{ij} \in K$.
This implies $\mu(e_{ij}) \in Z(\mathfrak{M}_n(r))$ for all $i \neq j$, by Lemma \ref{l1} (i). Hence $\mu$ is a central-valued linear map and $\theta$ is proper in this case.

\textbf{Case 2. $\hat{\mathcal{J}}\neq \emptyset$.}
If $\mathcal{J} = \emptyset$, then $\mathbf{n} = \hat{\mathcal{J}}$ and the multiplication of $\mathfrak{M}_n(r)$ is trivial. Hence we assume that $\mathcal{J} \neq \emptyset$ without loss of generality. 

By Lemma \ref{l28}, we have 
$$\theta(e_{ij}) = \lambda e_{ij} + \sum_{x,y \in \hat{\mathcal{J}} } C_{xy}^{ij} e_{xy}, \quad \text{for all } (i,j) \in \mathbf{n} \times \mathbf{n} \setminus \hat{\mathcal{J}} \times \hat{\mathcal{J}},$$ 
for some $\lambda \in K$. On the other hand, if $i,j \in \hat{\mathcal{J}}$, it follows from Corollaries \ref{l24} and \ref{l26} that
\[
\theta(e_{ij}) = \sum_{x,y \in \hat{\mathcal{J}} } C_{xy}^{ij} e_{xy},  \quad \text{for all } i,j\in \hat{\mathcal{J}}.
\]
Let us define a $K$-linear map $\mathcal{L}: \mathfrak{M}_n(r) \to \mathfrak{M}_n(r)$ by $\mathcal{L}(e_{ij}) = \lambda e_{ij}$, for all $i,j \in \mathbf{n}$. 

Let $\mu := \theta - \mathcal{L}$. 
If $(i,j) \in \mathbf{n} \times \mathbf{n} \setminus \hat{\mathcal{J}} \times \hat{\mathcal{J}}$, then $\mu(e_{ij}) = \sum_{x,y \in \hat{\mathcal{J}} } C_{xy}^{ij} e_{xy} \in Z(\mathfrak{M}_n(r))$ by Lemma \ref{l1} (ii). 
If $(i,j) \in \hat{\mathcal{J}} \times \hat{\mathcal{J}}$, then $\mu(e_{ij}) = -\lambda e_{ij} + \sum_{x,y \in \hat{\mathcal{J}} } C_{xy}^{ij} e_{xy} \in Z(\mathfrak{M}_n(r))$ by Lemma \ref{l1} (ii). Hence $\mu$ is a central-valued linear map and $\theta$ is standard in this case. Moreover, in this case, there exists
a standard commuting map which is improper.
\end{proof}


\begin{corollary} 
Every commuting map of ${\rm M}_n(K)$ is proper.
\end{corollary}

%
%
%

We end this note by extending Theorem \ref{l29} to a broader class of Munn's semigroup algebras.

\begin{remark}
Let $m,n$ be two positive integers. Let ${\rm M}_{m\times n}(K)$ be the set of all $m\times n$ matrices over $K$.
Given an $n\times m$ matrix $P$ over $K$, we define a multiplication ‘$\cdot$’ on ${\rm M}_{m\times n}(K)$ by
$A \cdot B=A P B$. It is clear that under the usual linear operations and the multiplication $\cdot$\, ,
${\rm M}_{m\times n}(K)$ becomes an associative $K$-algebra,
denoted by $\mathfrak{M}(m, n, P)$.

The algebra $\mathfrak{M}(m, n, P)$ was introduced independently by Munn \cite{Munn} when he studied the representation theory of semigroups, and by Brown \cite{Brown} when he studied the
representation theory of orthogonal groups.
Nowadays, $\mathfrak{M}(m, n, P)$ is known as a Munn's semigroup algebra or a generalized matrix algebra in the
sense of Brown, which is generally not 
a generalized matrix algebra
introduced by the third author and Wei \cite{XiaoWei}.
In fact, $\mathfrak{M}(m, n, P)$ is simple if and only if
it possesses an identity element \cite[Theorem]{Brown}.

Similar to Lemma \ref{lemma1}, we can get an isomorphism
of algebras
$\mathfrak{M}(m, n, P)\cong \mathfrak{M}(m, n, \left[\smallmatrix I_{r} & 0 \\ 0 & 0 \endsmallmatrix \right])$, where $r$ is the rank of the matrix $P$.
Then following the procedures done above for $\mathfrak{M}_n(r)$, we can show that
Theorem \ref{l29} also holds for $\mathfrak{M}(m, n, P)$,
i.e., every commuting map of $\mathfrak{M}(m, n, P)$
is standard.
\end{remark}
\smallskip
\noindent{\bf Funding.} 
The first author is supported by the National Natural Science Foundation of China (No.12401022).
The second author is
supported by the Natural Science Foundation of Fujian Province (No.2023J01126).

\smallskip
\noindent{\bf Data Availability Statement.} All data generated or analyzed during this study are included in this published article. No additional data are available.

\smallskip
\noindent{\bf Conflict of interest.} On behalf of all authors, the corresponding author states that there is no conflict of
interest.


\end{document}